\documentclass{amsart}[12pt]
\usepackage{amssymb,amsmath}
\usepackage{enumerate}
\usepackage[english]{babel}
\usepackage{color}

\newcommand{\ii }{\infty}

\newcommand{\su} {\subset}

\newtheorem{teo}{Theorem}[section]
\newtheorem{pro}{Proposition}[section]

\theoremstyle{definition}

\title[ Ergodic theorem for symmetric sequence spaces]{Uniform convergence in the individual ergodic theorem for symmetric sequence spaces}
\begin{document}
\date{June 17, 2017}
\begin{abstract}
It is proved that for any Dunford-Schwartz operator $T$ acting in the space $ l_\infty $ and for each $x\in c_0 $ there exists an element $\widehat {x} \in c_0 $ such that $\| \frac1n \sum_{k=0}^{n-1}T^k(x) - \widehat {x} \|_\infty \rightarrow 0$.
\end{abstract}

\author{Vladimir Chilin}
\address{National University of Uzbekistan\\ Tashkent,  700174, Uzbekistan}
\email{vladimirchil@gmail.com}

\author{Azizkhon Azizov}
\address{National University of Uzbekistan\\ Tashkent,  700174, Uzbekistan}
\email{azizov.07@mail.ru}

\maketitle

\section {Introduction}

Let $(\Omega,\mathcal A,\mu)$ be a $\sigma $-finite measure space, and let $L_p=L_p(\Omega,\mathcal A,\mu)$ be the classical Banach function space equipped with the norm $\|\cdot \|_p, \ 1\leq p\leq \infty $. A linear operator $T: L_1+ L_\infty \rightarrow  L_1+ L_\infty$ is called a {\it Dunford-Schwartz operator} (writing  $T\in DS$) if
$$
\| T(f)\|_1\leq \| f\|_1 \ \ \forall \ f\in   L_1 \text{\ and \ } \| T(f)\|_{\ii}\leq \| f\|_{\ii} \ \ \forall \ f\in L_\infty.
$$
It is known  that $T(L_p) \subset L_p$ for any $T\in DS$,  \ and   \ $\| T\|_{L_p\rightarrow L_p} \leq 1$   (see, for example \cite{bs}).

Dunford-Schwartz  individual ergodic theorem  \cite[Chapter VIII, Theorem VIII.6.6]{ds} states that for any  $T\in DS$ and $f \in L_p$, $1 \leq p <\infty $, there exists $\widehat {f} \in L_p$ such that
the averages
$$
A_n(T,f)=\frac1n \sum_{k=0}^{n-1}T^k(f)
$$
converge almost everywhere (a.e.) to  $\widehat {f}$.

By Egorov's theorem, in the case $\mu(\Omega) <\infty $, a.e. convergence of sequences of measurable functions coincides with almost uniform (a.u.) convergence. If $\mu(\Omega) = \infty$, then a.u. convergence is stronger than a.e. convergence. For example, if $(\mathbb N, 2^{\mathbb N}, \mu) $ is an atomic measure space, where $\mathbb N $ is the set of natural numbers, $2^{\mathbb N} $ is the $ \sigma$-algebra of subsets of $\mathbb N $, and $\mu(\{n\}) = 1 $ for each $ n \in \mathbb N $, then the sequence $ f_n = \{0, \dots, 0, n, 0, \dots \} \in L_\infty (\mathbb N, 2^{\mathbb N}, \mu) = l_\infty $ (the number $ n $ stands on the $ n $-th place) converges a.e., but does not converge a.u. It is clear that a.u. convergence in $l_\infty$ coincides with the convergence with respect to the  norm $ \|\cdot\|_\infty $.

In this paper we prove the following variant of Dunford-Schwartz  individual ergodic theorem
for $T\in DS$  acting in a fully symmetric sequence space $E \subset l_\infty$.

\begin{teo}\label{t11} Let  $E$ be a fully symmetric sequence space, $E \neq l_\infty$, and let $T\in DS$. Then for every $x\in E$ there is such  $\widehat {x} \in E$ that
$$\| \frac1n \sum_{k=0}^{n-1}T^k(x) - \widehat {x} \|_\infty \rightarrow 0.$$
\end{teo}

\section{Preliminaries}

Let $l_{\infty}$ (respectively, $c_0$)  be the Banach lattice of bounded (respectively, converging to zero) sequences $\{ \xi_n \}_{n = 1}^{\infty}$  of
 real numbers with respect to the norm $\|\{\xi_n\}\|_{\infty} = \sup\limits_{n \in \mathbb N} |\xi_n|$. Consider the measure space $(\mathbb N, 2^{\mathbb N}, \mu)$, for which  $L_{\infty}(\mathbb N, 2^{\mathbb N}, \mu)=l_{\infty}$ and
$$
L_{1}(\mathbb N, 2^{\mathbb N}, \mu)=l_1=\left \{\{ \xi_n \}_{n = 1}^{\infty}\subset \mathbb{R}:\ \|\{ \xi_n \} \|_1
=\sum_{n=1}^\infty |\xi_n|<\infty\right \}\subset l_{\infty}.
$$
If  $ x = \{\xi_n\}_{n = 1}^{\infty} \in l_{\infty}$, then a non-increasing rearrangement $x^*:[0,\infty) \rightarrow [0,\infty)$ of $x$  is defined by
$$
x^*(t) = \inf\{\lambda:\mu(\{|x|>\lambda\})\leq t\}, \ t \geq 0
$$
(see, for example, \cite[Ch. 2, Definition 1.5]{bs}). Therefore, a non-increasing rearrangement $x^*$ is identified with the sequence
$$
x^* = \{ \xi_n^* \}, \ \text{where}  \ \ \ \xi_n^*: = \inf\limits_{card (F)  <n} \ \sup\limits_{n \notin F} |\xi_n |, \ \ F \  \text{ a finite subset of} \ \ \mathbb N.
$$
If $ x = \{\xi_n\}_{n = 1}^{\infty} \in  c_0$, then there exists a bijection $\pi:\mathbb N \rightarrow \mathbb N$ such that $$\xi_{\pi(n)} = \xi_n^* \downarrow 0.$$

In what follows, an element $x \in l_\infty$ will be written in the form
$x=\{(x)_n\}_{n=1}^\infty $.
We use the following well-known property of non-increasing rearrangements \cite[Chapter II, \S 3]{kps}.

\begin{pro}\label{p21}
If $x_k,x \in l_\infty$ and $\|x_k-x\|_\infty \rightarrow 0$, then $\lim\limits_{k \to \infty}(x_k^*)_{n} = (x^*)_{n}$  for each $n \in \mathbb N$.
\end{pro}

A non-zero linear subspace  $E \subset l_{\infty}$ with a Banach norm $\|\cdot\|_{E}$ is called {\it  symmetric sequences  space} if conditions  $ y \in E$, $ x \in l_{\infty}, \  x^* \leq y^* $, imply that $ x \in E $ and $ \| x \|_E \leq \| y\|_E $. It is usually assumed that $\|\{1,0,0,\dots\}\|_E=1$.  For any symmetric sequences  space  $(E,\|\cdot\|_E)$ and $ x \in E $, we
have $\|x\|_{\infty} \leq\|x\|_E = \| \ |x| \ \|_E = \|x^*\|_E$.

If there exists $x \in E\setminus c_0$, then $x^*\geq\alpha \mathbf1$ for some $\alpha > 0$, where $\mathbf1 = \{1,1,...\}$. Consequently, $\mathbf1 \in E$ and $E = l_{\infty}$. Therefore, for every symmetric sequences  space $E$ we have
either $E \subset c_0$ or $E=l_{\infty}$. Below, we assume that $E \subset c_0$.

There are the following continuous embeddings \cite[Chapter 2, \S 6, Theorem 6.6]{bs}:
$$
(l_1,\|\cdot\|_1) \subset (E,\|\cdot\|_E) \subset (l_\infty,\|\cdot\|_\infty) \ \ \text{and} \ \ \|x\|_E \leqslant \|x\|_1 \ \ \text{for all} \ \ x \in l_1.
$$

Hardy-Littlewood-Polya partial order $x \prec y$ in  $l_\infty$ is defined as follows:
$$
(x=\{(x)_n\}_{n=1}^\infty \prec y=\{(y)_n\}_{n=1}^\infty) \ \Leftrightarrow \ \sum_{n=1}^k (x^*)_n\leq \sum_{n=1}^k
(y^*)_n \ \ \text{for all} \ \ k \in \mathbb N.
$$
A symmetric sequences  space $(E,\|\cdot\|_E)$ is called  {\it fully symmetric sequences  space} if conditions  $x \prec y$,
$x \in l_{\infty}, \ y \in E,$   imply that $x \in E$ and $\|x\|_E \leq \|y\|_E$.

The spaces $l_\infty, \ c_0$, \ $l_p=\{x=\{(x)_n\}_{n=1}^\infty \in c_0: \|x\|_p=(\sum_{n=1}^\infty
|(x)_n|^p)^\frac1p<\infty\}, \ 1 \leq p<\infty$, Orlicz and Lorentz sequence spaces are examples of
fully symmetric sequence spaces.

It should be noted that the Fatou property for space $l_p$  implies the following.
\begin{pro}\label{p22}\cite[Chapter IV, \S 3, Lemma 5]{KA}.
If
$$x_k=\{(x_k)_n\}_{n=1}^\infty \in l_p, \ \sup_{k \geq 1}\|x_k\|_p<\infty, \ x=\{(x)_n\}_{n=1}^\infty \in l_\infty \ \ \text{and} \ \ \|x_k-x\|_\infty \rightarrow 0,
$$
then $x \in l_p$.
\end{pro}
\section{Individual ergodic theorem in symmetric sequence spaces}

A linear operator $T: l_\infty \to  l_\infty$ is called a {\it Dunford-Schwartz operator} (writing  $T\in DS$) if
$$
\| T(x)\|_1\leq \| x\|_1 \ \ \forall \ \ x\in l_1 \ \ \text{and} \ \ \| T(x)\|_{\infty}\leq \| x\|_\infty \ \ \forall \ x \in l_\infty.
$$
It is know  that $T(E) \subset E $ and $ \| T \|_{E \to E} \leq 1 $ for every $T\in DS$ and  any fully symmetric sequences  space $ E $ (see, for example, \cite [Chapter II, \S 3] {kps}).
In addition,  Dunford-Schwartz operators have the following  property.

\begin{pro}\label{p31} \cite[Ch.4, \S 4.1]{kr}.
If $T \in DS$, then there exists a unique positive Dunford-Schwarz operator  $|T|:l_\infty \to  l_\infty$ such that $|T^k(x)|\leq |T|^k(|x|)$  for all $x\in l_\infty, \ k=1,2, \dots$
\end{pro}
Set
$$
A(n,T)=\frac1n \sum_{k=0}^{n-1}T^k, \ \ \widehat{A}(T,x)=\sup_{n \geq 1}|A(n,T)(x)|, \ T \in DS, \ x \in l_{\infty}.
$$
The following maximal ergodic estimate was proved in  \cite [Theorem 4.2] {ccl}.

\begin{teo}\label{t31} If \ $T\in DS$, \ $x\in l_p$, \ $1\leq p<\infty$, and $\alpha>0$, then
\begin{equation}\label{e1}
card \{ s \in N: (\widehat{A}(T,x))_s \geq \alpha \} \leq \left (2\frac {\|x\|_p}\alpha \right )^p. \end{equation}
\end{teo}

Theorem \ref{t31} in \cite {ccl} yields a new shorter proof of the individual ergodic theorem in the space $ l_p $.
Our main goal is, using the methods of \cite {ccl}, to obtain a version of individual ergodic theorem for $ T \in DS $ and  $ x \in c_0 $ with respect to $ \| \cdot \|_\infty $-convergence.

First, we establish this fact for the space $ l_p $:

\begin{teo}\label{t32} If $T\in DS$, then for each $x\in l_p$, $1\leq p<\infty$, there exists $\widehat {x} \in l_p $ such that $ \| \frac1n \sum_{k = 0}^ {n-1} T^k (x) - \widehat {x} \|_\infty \rightarrow 0$.
\end{teo}
\begin{proof}
We  will show  that the set
$$
l_p(T):=\{ x\in l_p: \text{\ there is \ } \widehat {x} \in l_p \text{\ such that \ } \| \frac1n \sum_{k=0}^{n-1}T^k(x) - \widehat {x} \|_\infty \rightarrow 0\}
$$
is closed in  $l_p$.

Let  $\{x_m\}\subset l_p(T), \ x\in l_p$ \ and \ $\| x-x_m\|_p \to 0$. Since $\{x_m\}_{m=1}^\infty \subset l_p(T)$,
it follows that there is an element $\widehat {x}_m \in l_p$ such that
$$ \| \frac1n \sum_{k = 0}^ {n-1} T^k (x_m) - \widehat {x}_m \|_\infty \rightarrow 0 \ \ \text{for all} \ \ m \in \mathbb N.
$$
Besides, by the classical individual ergodic theorem, the sequence $\{A(n,T)(x)\}$  converges coordinate-wise to some element $\widehat{x} \in l_p$.

Fix a positive integer $l$ and consider the set
$$
N_{m,l}=\{s \in \mathbb N: (\widehat{A}(|T|,|x-x_m|))_s\geq \frac1l \}.
$$
By (\ref{e1}) and Proposition \ref{p31}, we have
$$
card(N_{m,l}) \leq (2 l \|x-x_m\|_p)^p \to 0  \ \ \text{as} \ \ m \to \infty.
$$
Hence, there exists $m(l) \in \mathbb N$ such that
$N_{m,l}=\varnothing$ for every $m \geq m(l)$.
Therefore
$$\{s \in \mathbb N: (\widehat{A}(|T|,|x-x_m|))_s < \frac 1l \}=\mathbb N, \text{\ \ i.e.,} \ \
\|\widehat{A}(|T|,|x-x_m|)\|_\infty \leq \frac 1l
$$
for all \ $m \geq m(l)$.

Now, using
$$(\widehat{x})_s=\lim_{n\rightarrow \infty}  (A(n,T)(x))_s; \ \ \ (\widehat{x}_m)_s=\lim_{n\rightarrow \infty} (A(n,T)(x_m))_s, \ \ s \in \mathbb N,
$$
we obtain that
$$
\|\widehat{x}-\widehat{x}_m\|_\infty
=\sup_{s\geq 1}|\lim_{n\rightarrow \infty} ((A(n,T)(x))_s-(A(n,T)(x_m))_s)|\leq
$$
$$
\leq \sup_{s\geq 1} \sup_{n\geq 1}|(A(n,T)(x)-A(n,T)(x_m))_s|=\sup_{s\geq 1} \sup_{n\geq 1} |(A(n,T)(x-x_m))_s|=
$$
$$
= \sup_{s\geq 1}\sup_{n\geq 1}\left |\frac 1n\sum_{k=0}^{n-1}(T^k(x-x_m))_s\right| \leq \sup_{s\geq 1} \sup_{n\geq 1} \frac 1n \sum_{k=0}^{n-1}|(T^k(x-x_m))_s|\leq
$$
$$
\leq \sup_{s\geq 1} \sup_{n\geq 1} \frac 1n\sum_{k=0}^{n-1}(|T|^k(|x-x_m|))_s= \|\widehat{A}(|T|,|x-x_m|)\|_\infty \leq \frac 1l
$$
for all $l \in \mathbb N$ and $m\geq m(l)$. Thus  $\|\widehat{x}-\widehat{x}_m\|_\infty \to 0$ as $m \to \infty$. In addition,
$$
\sup_{n\geq 1}\|A(n,T)(x)-A(n,T)(x_m)\|_\infty \leq \|x-x_m\|_\infty  \leq  \|x-x_m\|_p \rightarrow 0.
$$
Since
$$\|A(n,T)(x)-\widehat{x}\|_\infty \leq \|A(n,T)(x)-A(n,T)(x_m)\|_\infty+$$
$$ +\|A(n,T)(x_m)-\widehat{x}_m\|_\infty+
\|\widehat{x}_m-\widehat{x}\|_\infty \to 0
$$
it follows that $\|A(n,T)(x)-\widehat{x}\|_\infty \to 0$ as $n \to \infty$. This means that
\begin{equation}\label{e2}
l_p(T) = \overline{l_p(T)}^{\|\cdot\|_p}.
\end{equation}

Show now that the sequence $\{A(n,T)(x)\}$ converges in the norm $\|\cdot\|_\infty$ for all $x \in (l_2, (\cdot,\cdot))$, where  $(\cdot,\cdot)$ is the standard scalar product in $l_2$.

Let  $N=\{T(z) - z: z \in l_2\}$ and $y \in N^\perp$. Then
$$0=(y,T(z)-z)=(T^*(y),z)-(y,z)=(T^*(y)-y,z)$$ for all $z \in l_2$, that is, $ T^*(y)=y$.
Since $T$ is a contraction in $l_2$, it follows that
\begin{equation}\label{e3}
\| T(y)-y\|^2_2 =\| T(y)\|^2_2-\| y\|^2_2\leq 0,
\end{equation}
that is,  $T(y)=y$. Therefore  $N^{\perp}\subset L=\{y\in l_2: T(y)=y\}$. Conversely, if $y \in L$, then replacing  in (\ref{e3}) $T$ by $T^*$, we see that $T^*(y)=y$. Hence $y \in N^{\perp}$ and  $N^{\perp}=L$.
This means that the set
$$
D=\{ y+(T(z)-z) :\  y, z \in l_2,\  T(y)=y\}
$$
is dense in  $l_2$.

It is clear that  $D \subset l_2(T)$, and, by (\ref{e2}),
$$
l_2 =  \overline{D}^{\|\cdot\|_2} \subset \overline{l_2(T)}^{\|\cdot\|_2} = l_2(T) \subset l_2,
$$
that is, $l_2=l_2(T)$.

Since the set $l_p\cap l_2$ is dense in $(l_p,\|\cdot\|_p)$ \ and \ $l_p\cap l_2 \subset l_p(T)= \overline{l_p(T)}^{\|\cdot\|_p}$, it follows that for each \  $x\in l_p$ \  there exists an element $\widehat {x} \in l_p $ such that $$ \|\frac1n \sum_{k = 0}^ {n-1} T^k (x) - \widehat {x}\|_\infty \rightarrow 0.$$
\end{proof}

Now we give a version of  the individual ergodic theorem for $ T \in DS $ and  $ x \in c_0 $ with respect to $ \| \cdot \|_\infty $-convergence.
\begin{teo}\label{t33} If $T\in DS$, then for each \  $x\in c_0$ there exists  $\widehat {x} \in c_0 $ such that $\|\frac1n \sum_{k = 0}^ {n-1} T^k (x) - \widehat {x}\|_\infty \rightarrow 0$.
\end{teo}
\begin{proof}
We can assume that  $0 \leq x =\{(x)_s\}_{s=1}^\infty \rightarrow 0$. It is clear that $x = y^{(k)} + z^{(k)}$, where
$$
y^{(k)}=\{(x)_1,(x)_2, \dots, (x)_{s_k},0,0,\dots\} \in l_1,
$$
$$\ z^{(k)}=\{0,0,\ldots,0,(x)_{s_k+1},\xi_{s_k+2}, \ldots\} \in c_0, \ \|z^{(k)}\|_\infty < \frac1k, \  k \in \mathbb N.
$$
By Theorem \ref{t32}, for every $k \in \mathbb N$ there exists an element $\widehat{y}^{(k)} \in l_1$ such that
$$
\|A(n,T)(y^{(k)})-\widehat{y}^{(k)}\|_\infty \rightarrow 0.
$$
In particular,
$$
\lim_{n \rightarrow \infty}  (A(n,T)(y^{(k)}))_s=(\widehat{y}^{(k)})_s
$$
for all $s \in \mathbb N$. Then, for a fixed $k$ and  $s \in \mathbb N$, we have
$$\limsup_{n \rightarrow \infty}(A(n,T)(x))_s - \liminf_{n \rightarrow \infty}(A(n,T)(x))_s=\limsup_{n \rightarrow \infty}(A(n,T)(y^{(k)}+z^{(k)}))_s-
$$
$$
-\liminf_{n \rightarrow \infty}(A(n,T)(y^{(k)}+z^{(k)}))_s =\limsup_{n \rightarrow \infty}  (A(n,T)(z^{(k)}))_s-
$$
$$
-  \liminf_{n \rightarrow \infty}  (A(n,T)(z^{(k)}))_s \leq 2\sup_n | (A(n,T)(z^{(k)}))_s| < \frac1k \rightarrow 0,
$$
hence $\limsup_{n \rightarrow \infty}(A(n,T)(x))_s = \liminf_{n \rightarrow \infty}(A(n,T)(x))_s
$ for all $s \in \mathbb N$. Therefore,
there exists a sequence of real numbers  $\widehat{x} = \{(\widehat{x})_s\}_{s=1}^\infty$
such that
$$
\lim_{n \rightarrow \infty}(A(n,T)(x))_s =(\widehat{x})_s \ \ \text{for all} \ \ s \in \mathbb N.
$$
Since
$$
\|A(n,T)(x)-\widehat{x}\|_\infty= \|A(n,T)(y^{(k)}+z^{(k)})-\widehat{x}\|_\infty\leq
$$
$$
\leq \|A(n,T)(y^{(k)})-\widehat{y}^{(k)}\|_\infty +\|A(n,T)(z^{(k)})\|_\infty+\|\widehat{y}^{(k)}-\widehat{x}\|_\infty\leq
$$
$$
\leq \|A(n,T)(y^{(k)})-\widehat{y}^{(k)}\|_\infty +\frac1k +\sup_{s \geq 1} |\lim_{n \rightarrow \infty} (A(n,T)(y^{(k)}))_s-\lim_{n \rightarrow \infty} (A(n,T)(x))_s|\leq
$$
$$
\leq \|A(n,T)(y^{(k)})-\widehat{y}^{(k)}\|_\infty +\frac1k + \sup_{s \geq 1} \sup_{n\geq 1}  |(A(n,T)(z^{(k)}))_s|\leq
$$
$$
\leq \|A(n,T)(y^{(k)})-\widehat{y}^{(k)}\|_\infty +\frac2k,
$$
it follows that $\|A(n,T)(x)-\widehat{x}\|_\infty \to 0$.

Now, using the inclusion $T(c_0) \subset c_0$ and completeness of the space $(c_0, \|\cdot\|_\infty)$,
we conclude that $\widehat{x} \in c_0$.
\end{proof}

The following theorem is a version of Theorem \ref{t33} for a fully symmetric  sequence space (see Theorem \ref{t11}).

\begin{teo}\label{t34} If \ $E$ is a fully symmetric  sequences space, \  $E \neq l_\infty$, and $T\in DS$, then for any  $x\in E$ there exists  $\widehat {x} \in E $ such that $\|\frac1n \sum_{k = 0}^ {n-1} T^k (x) - \widehat {x}\|_\infty \rightarrow 0$.
\end{teo}
\begin{proof}
Since  $E \neq l_\infty$ it follows that $E \subset c_0$ and by  Theorem  \ref{t33}  there exists  $\widehat {x} \in c_0$ such that $\|\frac1n \sum_{k = 0}^ {n-1} T^k (x) - \widehat {x}\|_\infty \rightarrow 0$.

Let us show that $\widehat{x} \in E$. By Proposition \ref{p21}, we have $((A(n,T)(x))^*)_s \to (\widehat{x}^*)_s$ for all $s \in \mathbb N$. Since $A(n,T) \in DS$, it follows that $A(n,T)(x)^* \prec x^*$ (see, for example, \cite[Ch.II, \S3, Sec.4]{kps}). Consequently,
$$
\sum_{s=1}^k (\widehat{x}^*)_s \leq \sup_{n\geq 1} \sum_{s=1}^k (A(n,T)(x)^*)_s \leq \sum_{s=1}^k (x^*)_s \ \ \text{for all } \ \ k \in \mathbb N,
$$
hence $\widehat{x}^* \prec x^*$. Since $E$ is fully symmetric, we get that $\widehat{x} \in E$.
\end{proof}

It should be noted that for an element $x \in l_\infty \setminus c_0$, Theorem \ref{t33}  is no longer valid:

\begin{teo}\label{t35} \cite[Theorem 3.3]{cl}. If $x \in l_\infty \setminus c_0$, then there exists \ $T\in DS$ such that  the averages $A(n,T)(x)$ do not converge coordinate-wise.
\end{teo}

We say that a fully  sequence space $E \subset l_\infty$ possesses the
{\it uniform individual ergodic theorem property}, writing $E\in (UIET)$, if for every $x \in E$ and $T\in DS$ the averages $A(n,T)(x)$  converge to some $\widehat x \in E$ with respect to the uniform norm $\|\cdot\|_{\infty}$.
Theorems \ref{t34} and \ref{t35} entail the following criterion.

\begin{teo}\label{t36} Let  $E$  be a fully symmetric sequence  space. The following conditions are equivalent:

$(i)$. $E\in (UIET)$;

$(ii)$. $E \su c_0$;

$(iii)$. $\mathbf 1\notin E$.
\end{teo}

\end{document}